\documentclass[11pt]{amsart}

\usepackage{amsmath,amsthm,verbatim,amssymb,amsfonts,amscd, graphicx,cite,url,enumerate,bbm}
\usepackage{graphics}
\usepackage[initials]{amsrefs}
\theoremstyle{plain}
\newtheorem{theorem}{Theorem}
\newtheorem{corollary}[theorem]{Corollary}
\newtheorem{lemma}[theorem]{Lemma}
\newtheorem{proposition}[theorem]{Proposition}

\theoremstyle{definition}
\newtheorem{definition}{Definition}
\theoremstyle{definition}
\newtheorem{example}[theorem]{Example}
\newtheorem*{remark}{Remark}
\newtheorem*{acknow}{Acknowledgements}

\DeclareRobustCommand{\stirling}{\genfrac\{\}{0pt}{}}

\begin{document}
	
	\title[Random Matrix-Valued Multiplicative Functions]{Random Matrix-Valued Multiplicative Functions and Linear Recurrences in Hilbert-Schmidt Norms of Random Matrices}
	\author{Maxim Gerspach}
	\thanks{The author was partially supported by DFG-SNF lead agency program grant 200020L\textunderscore175755}
	\address{ETH Z\"urich, Switzerland}
	\email{maxim.gerspach@math.ethz.ch}
	\maketitle
	
	\begin{abstract}
		We introduce the notion of a random matrix-valued multiplicative function, generalizing Rademacher random multiplicative functions to matrices. We provide an asymptotic for the second moment based on a linear recurrence property for Hilbert-Schmidt norms of sucessive products of random matrices. Moreover, we provide upper bounds for the higher even moments related to the generalized joint spectral radius.
	\end{abstract}
	
	\section{Introduction}
	
	A Rademacher random multiplicative function is a family $(f(n))_{n \in \mathbb{N}}$ (with the convention $0 \not \in \mathbb{N}$) of random variables taking values in $\{ \pm 1, 0 \}$ such that
	\begin{itemize}\setlength\itemsep{4pt}
		\item $n \mapsto f(n)$ is supported on squarefree integers,
		\item $(f(p))_{p \text{ prime}}$ are independent, each taking the values $ \pm 1$ with probability $\frac{1}{2}$ and 
		\item when $n = p_1 \cdots p_r$ is squarefree then we have $f(n) = f(p_1) \cdots f(p_r)$.
	\end{itemize}
	Moments of these functions have been studied in a great amount of detail. It is a classical fact that 
	\[ \mathbb{E}\bigg[ \Big( \sum_{n \le x} f(n) \Big)^2  \bigg] = \frac{6}{\pi^2} x + O(\sqrt{x}) \]
	and it was proven by Harper, Nikeghbali and Radziwi{\l}{\l} in \cite[Theorem 4]{HarperNR1} and independently by Heap and Lindqvist \cite[Theorem 4]{HeapLindqvist1} in the even case that for all integers $k\ge 3$ there exists a constant $C_k > 0$ such that
	\[ \mathbb{E}\bigg[ \Big( \sum_{n \le x} f(n) \Big)^k  \bigg] \sim C_k x^{k/2} (\log x)^{\binom{k}{2} - k}. \]
	
	In this work we will consider the following matrix-valued generalisation of Rademacher multiplicative functions.
	
	\begin{definition}
		Let $d \ge 1$ be an integer. A random matrix-valued multiplicative function is a family $(f(n))_{n \in \mathbb{N}}$ of random variables taking values in $\mathbb{C}^{d \times d}$ such that
		\begin{itemize}\setlength\itemsep{4pt}
			\item $n \mapsto f(n)$ is supported on squarefree integers,
			\item $(f(p))_{p \text{ prime}}$ are independent identically distributed (i.i.d.) and
			\item when $n = p_1 \cdots p_r$ is squarefree with $p_1 < \dots < p_r$ then we have $f(n) = f(p_1) \cdots f(p_r)$.
		\end{itemize} 
	\end{definition}
	
	\vspace{5pt}
	\noindent
	Our goal is to obtain estimates for the even moments
	\[ \mathbb{E} \Big[ \big \Vert \sum_{n \le x} f(n) \big \Vert_{HS}^{2 k} \Big], \]
	where $\Vert \cdot \Vert_{HS}$ denotes the Hilbert-Schmidt norm defined by $\Vert A \Vert_{HS}^2 = \mathrm{Tr}(A^* A)$ for $A \in \mathbb{C}^{d \times d}$. 

	In section \ref{SecMom} we will prove the following estimate for the second moment based on a linear recurrence property of the Hilbert-Schmidt norm, which will be the subject of section \ref{HSN}.
	
	\begin{theorem}\label{diag}
		Let $d \ge 1$ be an integer, let $X$ be a $\mathbb{C}^{d \times d}$-valued random variable and let $f$ be the associated matrix-valued multiplicative function. Suppose that $\mathbb{E} X = 0$ and
		\[\mathbb{E}\left[ \Vert X \Vert_{HS}^2 \mathbbm{1}(\Vert X \Vert_{HS}^2 > R) \right] \xrightarrow{R \to \infty} 0,\] 
		where $\mathbbm{1}(E)$ denotes the characteristic function of an event $E$. Define
		\begin{align*}
		T : \mathbb{C}^{d \times d} &\to \mathbb{C}^{d \times d},\\
		A &\mapsto \mathbb{E}[ X^* A X ],
		\end{align*}
		let $l:=d^2$ and assume that $T$ is diagonalizable with eigenvalues $\lambda_1, \dots, \lambda_l$ arranged in descending order according to their real parts. Then for any $N \in \mathbb{N}$ there are constants $C_{i,m}, \; m = 1, \dots, N, \, i = 1, \dots, l$ such that
		\[ \mathbb{E} \Big[ \big \Vert \sum_{n \le x} f(n) \big \Vert_{HS}^2 \Big] = x \sum_{m=1}^N \sum_{i=1}^l C_{i,m} (\log x)^{\lambda_i-m} + O \left( x (\log x)^{\lambda_1 - N - 1} \right) \]
		holds for all $x \ge 2$.
	\end{theorem}

	Our argument also extends to the case when $T$ is not diagonalizable, even though our estimate becomes less precise in this case. The exact statement without the assumption of diagonalizability will be given and proven in section \ref{SecMom}.
	
	Section \ref{HighMom} will be devoted to proving an upper bound for higher moments that will be related to what is known as the generalized joint spectral radius.
	
	\begin{acknow}
		The author would like to thank Emmanuel Kowalski and Jori Merikoski for helpful discussions and comments on earlier drafts of this paper.
	\end{acknow}
	
	\section{A Linear Recurrence for Hilbert-Schmidt Norms}\label{HSN}
	
	The goal of this section is to prove the following result which may be of independent interest.
	
	\begin{theorem}\label{LinRec}
		Let $d,k \ge 1$ be fixed integers. Suppose that $X, X_1, X_2, \dots$ is a sequence of i.i.d. $\mathbb{C}^{d \times d}$-valued random variables such that
		\[ \mathbb{E} \left[ \Vert X \Vert_{HS}^{2k} \mathbbm{1}( \Vert X \Vert_{HS}^{2k} >R )  \right] \to 0 \]
		as $R \to \infty$, and define
		\[ a_n := a_n^{(2k)} := \mathbb{E} \left[ \Vert X_1 \cdots X_n \Vert_{HS}^{2k} \right]. \]
		Then the sequence $(a_n)_n$ satisfies a linear recurrence of length
		\begin{equation*}
		l^\mathbb{C} := \binom{k+d^2-1}{k}.
		\end{equation*}
		If the random variables are in fact $\mathbb{R}^{d \times d}$-valued, then the sequence $(a_n)_n$ satisfies a linear recurrence of length
		\begin{equation*}
		l^\mathbb{R} := \binom{k+\binom{d+1}{2}-1}{k}.
		\end{equation*}
		
	\end{theorem}

	Before proving this Theorem, we first recall the following standard fact about linear recurrences.
	
	\begin{lemma}\label{GenLinRec}
		Let $(a_n)_n$ be a sequence of complex numbers, and let 
		\[p(x) = x^l + c_1 x^{l-1} + \dots + c_l = (x-\lambda_1)^{m_1} \cdots (x-\lambda_t)^{m_t}\] 
		be a polynomial such that
		\[ a_{n+l} + c_1 a_{n+l-1} + \dots + c_l a_n = 0 \]
		holds for all $n$, where $\lambda_1, \dots, \lambda_t$ are distinct complex numbers. Then there exist unique polynomials $g_i$ of degrees $< m_i$ for $i=1,\dots, t$ such that
		\[ a_n = g_1(n) \lambda_1^n + \dots + g_t(n) \lambda_t^n. \]
	\end{lemma}

	\begin{proof}[Proof of Theorem \ref{LinRec}]
		Let $\mu$ be the law of $X$. Let $S_d$ denote the space of complex-symmetric $d \times d$ matrices and let
		\[ V^\mathbb{C}:= \mathrm{Sym}^k(\mathbb{C}^{d \times d}) \quad \text{and} \quad V^\mathbb{R} := \mathrm{Sym}^k(S_d). \]
		Note that $V^\mathbb{C}$ resp. $V^\mathbb{R}$ is a complex vector space of dimension $l^\mathbb{C}$ resp. $l^\mathbb{R}$.
		
		Further, define
		\begin{align*}
		T^\mathbb{C} : V^\mathbb{C} &\to V^\mathbb{C}, \\
		v &\mapsto \mathbb{E}\big[ (X^*)^{\otimes k} v X^{\otimes k} \big]
		\end{align*}
		and let $T^\mathbb{R}$ be its restriction to $V^\mathbb{R}$ whenever $X$ is real-valued.
		
		In the following, we will shorten notation by writing $V, \, T$ and $l$ in place of the corresponding real and complex objects whenever a statement holds in both cases. We will adopt this convention for objects defined later on. Moreover, we will write $\mathbb{K}$ as a placeholder for $\mathbb{R}$ and $\mathbb{C}$.
		
		Finally, we denote by
		\[ p_T(x) = x^l + c_1 x^{l-1} + \dots + c_l \]
		the characteristic polynomial of $T$.
		
		\vspace{10pt}
		\noindent
		\textbf{Part 1: $\mu$ has finite support.} 
		
		\noindent
		Inductively applying the mixed-product identity of the Kronecker product \[(A_1 \otimes A_2)(A_3 \otimes A_4) = A_1 A_3 \otimes A_2 A_4\] for $A_1, A_2, A_3, A_4 \in \mathbb{C}^{d \times d}$ implies that
		\begin{align*} 
		a_n = \mathbb{E} \left[ \mathrm{Tr}(X_n^* \cdots X_1^* X_1 \cdots X_n )^k \right] &= \mathbb{E} \left[ \mathrm{Tr} \left( (X_n^* \cdots X_1^* X_1 \cdots X_n)^{\otimes k} \right) \right] \\
		&= \mathrm{Tr} \left( \mathbb{E} \left[ (X_n^*)^{\otimes k} \cdots (X_1^*)^{\otimes k} X_1^{\otimes k} \cdots X_n^{\otimes k} \right] \right).
		\end{align*}
		We claim that the sequence $(a_n)_n$ satisfies the recurrence defined by the characteristic polynomial of $T$, i.e. for all $n \in \mathbb{N}$ we have
		\[ a_{n+l} + c_1 a_{n+l-1} + \dots + c_l a_n = 0. \]
		
		In order to see this, assume first that $T$ is diagonalizable. Then we can write the identity $I:= I_d^{\otimes k} \in V$ as a linear combination of eigenvectors of $T$, i.e. there are $\lambda_1, \dots, \lambda_l, \, \alpha_1, \dots, \alpha_l \in \mathbb{C}$ and non-zero $v_1, \dots, v_l \in V$
		such that
		\[ T v_i = \lambda_i v_i \quad \text{ and } \quad I = \sum_{i=1}^l \alpha_i v_i. \]
		This implies that
		\begin{align*}
		a_n &= \mathrm{Tr} \left( \mathbb{E} \left[ (X_n^*)^{\otimes k} \cdots (X_1^*)^{\otimes k} \left( \sum_{i=1}^l \alpha_i v_i \right) X_1^{\otimes k} \cdots X_n^{\otimes k} \right] \right) \\
		&=\sum_{i=1}^l \alpha_i \lambda_i \mathrm{Tr} \left( \mathbb{E} \left[ (X_n^*)^{\otimes k} \cdots (X_2^*)^{\otimes k} v_i X_2^{\otimes k} \cdots X_n^{\otimes k} \right] \right).
		\end{align*}
		Inductively, we obtain
		\[a_n = \sum_{i=1}^l \alpha_i \mathrm{Tr}(v_i) \lambda_i^n, \]
		so that the sequence $(a_n)_n$ indeed satisfies the characteristic polynomial of $T$ under the assumption that this operator is diagonalizable.
		
		Now fix $m$ and weights $p_1, \dots, p_m > 0$ with $\sum p_i = 1$. Given $B_1, \dots, B_m \in \mathbb{K}^{d \times d}$, define the (finitely supported) probability measure
		\[ \mu := \sum_{i=1}^m p_i \delta_{B_i}, \]
		where $\delta_B$ denotes the Dirac measure at $B \in \mathbb{K}^{d \times d}$.
		Set $\beta = 1$ resp. $2$ when $X$ is real- resp. complex-valued. Endowing $ (\mathbb{K}^{d \times d})^m = \mathbb{R}^{\beta md^2} $ with the Zariski topology, we claim that the set \[M:=\{(B_1, \dots,B_m) \in (\mathbb{K}^{d \times d})^m : T \text{ diagonalizable} \} \subseteq \mathbb{R}^{\beta m d^2} \] is dense. \footnote{We do not claim that this is a dense condition in $\mathbb{C}^{md^2}$ in the complex case, but only in $\mathbb{R}^{2md^2}$.} First, note that this set is non-empty: Choose $B_1 = \dots = B_m = I$ all to be the identity matrix. Then $T$ is the identity on $V$, hence diagonalizable.
		
		The next step is to prove that $M$ is Zariski-open. But this follows from the fact that the map
		\begin{align*}
		\tau \, : \, \mathbb{R}^{\beta md^2} &\to \mathrm{End}(V), \\
		(B_1, \dots, B_m) &\mapsto T
		\end{align*}
		is polynomial in the entries of the $B_i$ and diagonalizability of $T$ is an open condition on the right-hand side. Hence, $M$ is indeed a dense set. Note that $\tau$ is not polynomial on $\mathbb{C}^{dm^2}$ in the complex case.
		
		Lastly, consider for fixed $n$ the composition of maps
		\[ \mathbb{R}^{\beta m d^2} \to \mathrm{End}(V) \times \mathbb{R}^{l+1} \to \mathbb{C}^l \times \mathbb{R}^{l+1} \to \mathbb{C} \]
		given by
		\begin{align*} (B_1, \dots, B_m) &\mapsto (T, (a_{n+l},\dots, a_n)), \\ (T, (b_l,\dots, b_0)) &\mapsto (p_T,(b_l,\dots, b_0)), \\ ((c_1, \dots, c_l),(b_l,\dots, b_0)) &\mapsto b_l + c_1 b_{l-1} + \dots + c_l b_0,
		\end{align*}
		where in the second map we send an operator to its characteristic polynomial viewed as a vector in its coefficients. It is clear that each of these maps is continuous, and we know that their composition
		\[ (B_1, \dots, B_m) \mapsto a_{n+l} + a_{n+l-1} c_1 + \dots + a_n c_l \]
		vanishes on the dense set $M$, hence everywhere, which settles Part 1.
		
		\vspace{10pt}
		\noindent
		\textbf{Part 2: $\mu$ has compact support.}
		
		\noindent
		Our goal is to show that the equation
		\[ a_{n+l} + c_1 a_{n+l-1} + \dots + c_l a_n = 0 \]
		still holds for all $n$. Let $K = \text{supp } \mu$ and let $(\mu_m)_m$ be a sequence of probability measures with finite support contained in $K$ such that $\mu_m \to \mu$ weakly, i.e. for all continuous bounded functions $f : \mathbb{K}^{d \times d} \to \mathbb{R}$ we have
		\[ \int_{\mathbb{K}^{d \times d}} f d \mu_m \to \int_{\mathbb{K}^{d \times d}} f d \mu \]
		as $m \to \infty$. Let $(X_{n,m})_n$ be i.i.d. sequences of random variables distributed according to $\mu_m$, let
		\[ a_{n,m} = \mathbb{E} \big[ \Vert X_{1,m} \cdots X_{n,m} \Vert_{HS}^{2k} \big], \]
		and similarly define $T_m$ and $c_{i,m}$ w.r.t. $\mu_m$. Since the measures $\mu_m$ have finite support, we know that
		\[ a_{n+l,m} + c_{1,m} a_{n+l-1,m} + \dots + c_{l,m} a_{n,m} = 0 \]
		holds for all $n,m$. It thus suffices to show that for any fixed $n$ and $i$ we have $a_{n,m} \to a_n$ and $c_{i,m} \to c_i$ as $m \to \infty$.
		
		It is a standard fact that the weak convergence of $(\mu_m)_m$ implies the weak convergence of the product measures $\mu_m^{\otimes n} \to \mu^{\otimes n}$. Moreover, we have
		\begin{align*} a_n &= \int_{(\mathbb{K}^{d \times d})^n} \Vert A_1 \cdots A_n \Vert_{HS}^{2k} d \mu^{\otimes n}(A_1, \dots, A_n), \\ a_{n,m} &= \int_{(\mathbb{K}^{d \times d})^n} \Vert A_1 \cdots A_n \Vert_{HS}^{2k} d \mu_m^{\otimes n}(A_1, \dots, A_n). \end{align*}
		Letting $f : (\mathbb{K}^{d \times d})^n \to \mathbb{R}$ be a bounded continuous function which coincides with $(A_1, \dots, A_n) \mapsto \Vert A_1 \cdots A_n \Vert_{HS}^{2k}$ on $K^{\times n}$ implies the convergence $a_{n,m} \to a_n$ as $m \to \infty$ for any fixed $n$.
		
		To show the convergence of $c_{i,m}$, note that it suffices to show that each entry of $T_m$ in some fixed basis converges to the corresponding entry of $T$. But for this, in turn, it suffices in both the real and the complex case to show the same property for the extended operator
		\begin{align*}
		\tilde T : (\mathbb{C}^{d \times d})^{\otimes k} &\to (\mathbb{C}^{d \times d})^{\otimes k} \\
		A &\mapsto \mathbb{E}\big[ (X^T)^{\otimes k} A X^{\otimes k} \big]
		\end{align*}
		and the corresponding operators $\tilde T_m$, since $T$ and $T_m$ are just restrictions of these operators to a common invariant subspace. Let us take the standard basis given by $e_{i_1 j_1} \otimes \dots \otimes e_{i_k j_k}$ with $i_1, j_1, \dots, i_k, j_k \in \{1, \dots, d\}$, where $e_{i j} \in \mathbb{C}^{d \times d}$ is the matrix with entry $i j$ being $= 1$ and the rest $ = 0$. One verifies that
		\begin{align*} (\tilde T (e_{i_1 j_1} \otimes \dots \otimes e_{i_k j_k}))_{e_{i_1' j_1'} \otimes \dots \otimes e_{i_k' j_k'} } &= \mathbb{E} [ \overline{X_{i_1 i_1'}} X_{j_1 j_1'} \cdots \overline{X_{i_k i_k'}} X_{j_k j_k'} ] \\ &= \int_{\mathbb{K}^{d \times d}} \overline{A_{i_i i_1'}} A_{j_1 j_1'} \cdots \overline{A_{i_k i_k'}} A_{j_k j_k'} d \mu(A),
		\end{align*}
		where in the complex case the integral is taken over real and imaginary part separately; analogous statements hold for $\tilde T_m$. Again taking bounded continuous functions $f : \mathbb{K}^{d \times d} \to \mathbb{R}$ which coincide with real resp. imaginary part of $A \mapsto \overline{A_{i_i i_1'}} A_{j_1 j_1'} \cdots \overline{A_{i_k i_k'}} A_{j_k j_k'}$
		on $K$ implies the claim.
		
		\vspace{10pt}
		\noindent
		\textbf{Part 3: The general case.}
		
		Let $R > 0$ be sufficiently large so that $\mu(B_R(0)) > 0$. Define the (conditional) probability measure
		\[ \mu_c^R(M) := \frac{\mu(M \cap B_R(0))}{\mu(B_R(0))}.  \]
		Denoting by $(X_n^R)_n$ a family of i.i.d. random variables corresponding to $\mu_c^R$, we can set
		\[ a_n^R := \mathbb{E} \big[ \Vert X_1^R \cdots X_n^R \Vert_{HS}^{2 k} \big] \]
		and similarly $T^R$ and $c_i^R$. Since $\mu_c^R$ has compact support, we know that
		\[ a_{n+l}^R + c_1^R a_{n+l-1}^R + \dots + c_l^R a_n^R = 0 \]
		holds for all $n$. It thus suffices to show that for any fixed $n$ and $i$ we have $a_n^R \to a_n$ and $c_i^R \to c_i$ as $R \to \infty$.
		
		We have
		\begin{align*}
		a_n^R &= \int_{(\mathbb{K}^{d \times d})^n} \Vert A_1 \cdots A_n \Vert_{HS}^{2k} \, d (\mu_c^R)^{\times n}(A_1, \dots, A_n) \\
		&= \frac{1}{\mu(B_R(0))^n} \int_{(\mathbb{K}^{d \times d})^n} \Vert A_1 \cdots A_n \Vert_{HS}^{2k} \mathbbm{1}(\Vert A_1 \Vert_{HS}^{2 k} \le R ) \cdots \mathbbm{1}(\Vert A_n \Vert_{HS}^{2 k} \le R ) \, d\mu^{\times n}(A_1,\dots,A_n)
		\end{align*}
		Since $\mu(B_R(0)) \to 1$ as $R \to \infty$, it suffices to show that this integral converges to $a_n$ as $R \to \infty$. But
		\begin{align*}
		&\quad \; \left| a_n - \int_{(\mathbb{K}^{d \times d})^n} \Vert A_1 \cdots A_n \Vert_{HS}^{2k} \mathbbm{1}(\Vert A_1 \Vert_{HS}^{2 k} \le R ) \cdots \mathbbm{1}(\Vert A_n \Vert_{HS}^{2 k} \le R ) \, d\mu^{\otimes n}(A_1,\dots,A_n) \right| \\
		&\le n \int_{(\mathbb{K}^{d \times d})^n} \Vert A_1 \cdots A_n \Vert_{HS}^{2k} \mathbbm{1}(\Vert A_1 \Vert_{HS}^{2 k} > R ) \, d\mu^{\otimes n}(A_1,\dots,A_n) \\
		&\le n \int_{\mathbb{K}^{d \times d}} \Vert A_1 \Vert_{HS}^{2 k} \mathbbm{1}(\Vert A_1 \Vert_{HS}^{2 k} > R ) \, d\mu(A_1) \int_{\mathbb{K}^{d \times d}} \Vert A_2 \Vert_{HS}^{2 k} \, d\mu(A_2) \cdots \int_{\mathbb{K}^{d \times d}} \Vert A_n \Vert_{HS}^{2 k} \, d\mu(A_n) \xrightarrow{R \to \infty} 0
		\end{align*}
		by assumption.
		
		For the convergence of $c_i^R$ to $c_i$, it again suffices to show that every entry of $T^R$ converges to the corresponding entry of $T$ in some fixed basis. Again, it suffices to show this for the extended operators $\tilde T$ and $\tilde T^R$ defined in the obvious way. But we have
		\begin{align*}
		&\left|(\tilde T - \frac{1}{\mu(B_R(0))^{2 k}}\tilde T^R)(e_{i_1 j_1} \otimes \dots \otimes e_{i_k j_k}))_{e_{i_1' j_1'} \otimes \dots \otimes e_{i_k' j_k'} } \right| \\ = &\left| \int_{\mathbb{K}^{d \times d}} \overline{A_{i_1 i_1'}} A_{j_1 j_1'} \cdots \overline{A_{i_k i_k'}} A_{j_k j_k'} \mathbbm{1}(\Vert A \Vert_{HS}^{2 k} > R) \, d\mu(A) \right|
		\le \int_{\mathbb{K}^{d \times d}} \Vert A \Vert_{HS}^{2k} \mathbbm{1}(\Vert A \Vert_{HS}^{2 k} > R) \, d\mu(A) \to 0
		\end{align*}
		as $R \to \infty$, hence the claim.
	\end{proof}

	\begin{remark}
		The idea of reducing matrix dimensions by looking at symmetric algebras in a similar context of Theorem \ref{LinRec} has been considered in \cite{BlondelN1,ParriloJ1} related to Kronecker and semidefinite lifting.
		
		One might be interested in the optimality of $l$. In the real case, we can in fact prove that $l^\mathbb{R}$ is optimal in the sense that for all $d, k \ge 1$ there exists $X$ such that the sequence $(a_n)_n$ does not satisfy a linear recurrence of any shorter length. 
		
		In this case, it in fact suffices to take $X$ to be a deterministic distribution supported in a single point $A$. Let $\lambda_1, \dots, \lambda_d$ be the eigenvalues of $A$ and assume for simplicity that they are algebraically independent. One verifies that the eigenvalues of
		\begin{align*}
		S_d &\to S_d, \\
		B &\mapsto B^T A B
		\end{align*}
		are given by $\lambda_i \lambda_j$ for $1 \le i \le j \le d$, which we will denote by $\mu_1, \dots, \mu_{d'}$ with $d' = \binom{d+1}{2}$. Moreover, it is elementary to see that the eigenvalues of $T$ are then given by $\mu_{i_1} \cdots \mu_{i_k}$ for $1 \le i_1 \le \dots \le i_k \le d'$, and that they are pairwise distinct. A generic choice of $A$ will satisfy $\alpha_i \mathrm{Tr}(v_i) \ne 0$ for all $i$, which then implies the claim.
		
		In the complex case, taking a deterministic $X$ and doing the same construction as in the real case gives an operator $T$ which can have at most
		\[ {\tilde l}^{\mathbb{C}} = \binom{k + d - 1}{k}^2 \]
		distinct eigenvalues. The above argument does prove that there are $X$ such that $(a_n)_n$ satisfies a linear recurrence of no shorter length than ${\tilde l}^{\mathbb{C}}$, but it does not give optimality of $l^\mathbb{C}$. For this, one would need to take a more complicated $X$, for which it is significantly more difficult to explicitly compute the eigenvalues of $T$. Nonetheless, numerical evidence in this case does suggest that $l^\mathbb{C}$ might still be optimal.
		
		The eigenvalue of $T$ of largest real part is essentially the generalized joint spectral radius of $X$. More details on this can be found in section \ref{HighMom}.
\end{remark}

	For further reference, we would like to record the following
	
	\begin{corollary}\label{LinRec2}
		Let $X, X_1, X_2, \dots$ be i.i.d. $\mathbb{C}^{d \times d}$-valued random variables with
		\[ \mathbb{E} \left[ \Vert X \Vert_{HS}^2 \mathbbm{1}(\Vert X \Vert_{HS}^2 > R) \right] \xrightarrow{R \to \infty} 0 \]
		and define
		\[ a_n := \mathbb{E} \left[ \Vert X_1 \cdots X_n \Vert_{HS}^2 \right]. \]
		Let
		\begin{align*} T : \mathbb{C}^{d \times d} &\to \mathbb{C}^{d \times d}, \\
		A &\mapsto \mathbb{E}\left[ X^* A X \right]
		\end{align*}
		and denote by
		\[ p_T(x) = x^l + c_1 x^{l-1} + \dots + c_l \]
		the characteristic polynomial of $T$, where $l := d^2$ is the dimension of $\mathbb{C}^{d \times d}$. Then for any $n \in \mathbb{N}$, we have
		\begin{align*}\label{LinRecEq} 
		a_{n+l} + c_1 a_{n+l-1} + \dots + c_l a_n = 0. 
		\end{align*}
	\end{corollary}

	\section{Second-Moment Estimate for Random Matrix-Valued Multiplicative Functions}\label{SecMom}

	We will fix the following notation: We set
	\[ P(s,z) := \prod_p \left(1+\frac{z}{p^s}\right) \left( 1- \frac{1}{p^s} \right)^z\]
	and
	\[ F(s,z) := \frac{P(s,z)}{\Gamma(z)}, \]
	as well as $P(z):= P(1,z)$ and $F(z) := F(1,z)$.
	
	\subsection{The Diagonalizable Case}

	Using Theorem \ref{LinRec}, we are now in a position to prove Theorem \ref{diag}.
	
	\begin{proof}[Proof of Theorem \ref{diag}]
		We have
		\[ \mathbb{E} \Big[ \big \Vert \sum_{n \le x} f(n) \big \Vert_{HS}^2 \Big] = \sum_{n_1, n_2 \le x} \mathbb{E}\big[\mathrm{Tr}\big(f(n_1)^* f(n_2) \big) \big] = \sum_{n \le x} \mathrm{Tr} \big( \mathbb{E}[f(n)^* f(n)] \big). \]
		By definition of $f$, the contribution of squarefree $n$ to this sum depends only on $\omega(n)$ and is given by
		\[ a_{\omega(n)} = \mathrm{Tr} \big( \mathbb{E}[ X_{\omega(n)}^* \cdots X_1^* X_1 \cdots X_{\omega(n)} ] \big) \]
		for i.i.d. random variables $X, X_1, \dots, X_{\omega(n)}$, where $a_{\omega(n)}$ is defined as in Corollary \ref{LinRec2}. But this implies
		\[ \mathbb{E} \Big[ \big \Vert \sum_{n \le x} f(n) \big \Vert_{HS}^2 \Big] = \sum_{n \le x} \mu^2(n) a_{\omega(n)} = \sum_{i=1}^l \alpha_i \mathrm{Tr}(v_i) \sum_{n \le x} \mu^2(n) \lambda_i^{\omega(n)}. \]
		It thus remains to prove the following
		
		\begin{proposition}\label{AsympDiag}
			For any $N \in \mathbb{N}$ and $z \in \mathbb{C}$ there are explicit constants $C_1, \dots, C_N$ (depending on $z$) such that
			\[ \sum_{n \le x} \mu^2(n) z^{\omega(n)} = x \sum_{m=1}^N C_i (\log x)^{z-m} + O \left( x (\log x)^{z-N-1} \right). \]
			For example, we have
			\[\sum_{n \le x} \mu^2(n) z^{\omega(n)} = F(z) x (\log x)^{z-1} + \frac{(\gamma z - 1) P(z) + P_s(z)}{\Gamma(z-1)} x (\log x)^{z-2} + O \left( x (\log x)^{z-3} \right), \]
			where
			$P_s(z)$ denotes the derivative of $P(s,z)$ w.r.t. $s$ evaluated at $s=1$, and where $\gamma$ is the Euler-Mascheroni constant. In fact, the error terms are uniform over $|z| < A$.
		\end{proposition}
	
	This follows from \cite[Theorem, p. 188]{Dixon1} by setting
	\[ a_z(n) := \mu^2(n) z^{\omega(n)}\]
	in the notation there, so that
	\[ f(s,z) := \sum_{n \ge 1} \frac{a_z(n)}{n^s} = \sum_{n \ge 1} \frac{\mu^2(n) z^\omega(n)}{n^s}, \]
	which gives
	\[ g(s,z) := (s-1)^z f(s,z) = \left[ (s-1) \zeta(s) \right]^z \zeta(s)^{-z} f(s,z) = \left[ (s-1) \zeta(s) \right]^z P(s,z). \]
	Taylor expansion of $(s-1) \zeta(s)$ around $s=1$ and application of the Binomial Theorem quickly yields, for example,
	\[\left[ (s-1) \zeta(s) \right]^z = 1 + \gamma z (s-1) + O \left( (s-1)^2 \right).\]
	This quickly gives the second part of the assertion.
	
	Since we can compute an arbitrary number of terms in this Taylor series and also the one for $P$ around $s=1$, this gives the first claim by \cite[Theorem, p. 188]{Dixon1}.
	This concludes the proof of Theorem \ref{diag}.
	
	\end{proof}

	\begin{remark}
		The proof shows that the constants $C_{i,m}$ in Theorem \ref{diag} are explicit. Let $v_1, \dots, v_l \in V$ be the eigenvectors of $T$ associated to $\lambda_1, \dots, \lambda_l$, and let $\alpha_1, \dots, \alpha_l \in \mathbb{C}$ be such that $I = \sum \alpha_i v_i$. Then, for example, we have
		\[ C_{i,1} = \alpha_i \mathrm{Tr}(v_i) F(\lambda_i) \]
		and
		\[ C_{i,2} = \alpha_i \mathrm{Tr}(v_i) \frac{(\gamma \lambda_i - 1) P(\lambda_i) + P_s(\lambda_i)}{\Gamma(\lambda_i-1)}.  \]
		By the methods outlined in the proof of Proposition \ref{AsympDiag} one can compute arbitrarily many such constants.
		
		We also remark that if $X$ is real-valued then by Theorem \ref{LinRec} we can restrict $T$ to $S_d$ and set $l=\binom{d+1}{2}$.
	\end{remark}
	
	\begin{example}\label{SL2}
		Let $X$ be the uniform distribution on the set
		\[ S = \left\{ \pm \begin{pmatrix} 1 & 0 \\ 0 & 1 \end{pmatrix}, \pm \begin{pmatrix} 1 & 1 \\ 0 & 1 \end{pmatrix}, \pm \begin{pmatrix} 1 & -1 \\ 0 & 1 \end{pmatrix}, \pm \begin{pmatrix} 0 & 1 \\ -1 & 0 \end{pmatrix} \right\} \]
		and let $f$ be the associated matrix-valued multiplicative function. Then we have $\mathrm{Im} f = \mathrm{SL}_2(\mathbb{Z})\cup \{0\}$ almost surely. If
		\begin{align*}
		T : S_d &\to S_d, \\
		A &\mapsto \mathbb{E}[X^T A X]
		\end{align*}
		then it is verified by evaluating at $\begin{pmatrix}
		1 & 0 \\ 0 & 0
		\end{pmatrix}, \, \begin{pmatrix}
		0 & 1 \\ 1 & 0
		\end{pmatrix}$ and $\begin{pmatrix}
		0 & 0 \\ 0 & 1
		\end{pmatrix}$
		that $T$ can be represented by the matrix
		\[ T = \frac{1}{4} \begin{pmatrix}
		3 & 0 & 1 \\
		0 & 2 & 0 \\
		3 & 0 & 3
		\end{pmatrix} \]
		with eigenvalues
		\[ \lambda_1 = \frac{3+\sqrt{3}}{4}, \; \lambda_2 = \frac{3-\sqrt{3}}{4} \text{ and } \lambda_3 = \frac12 \]
		and eigenvectors
		\[ v_1 = \begin{pmatrix}
		1 & 0 \\ 0 & \sqrt{3}
		\end{pmatrix}, \; v_2 = \begin{pmatrix}
		-1 & 0 \\ 0 & \sqrt{3}
		\end{pmatrix} \text{ and } v_3 = \begin{pmatrix}
		0 & 1 \\ 1 & 0
		\end{pmatrix}. \]
		Moreover, we can write the identity matrix as $I_2 = \sum_{i=1}^3 \alpha_i v_i$ with
		\[ \alpha_1 = \frac{3+\sqrt{3}}{6}, \; \alpha_2 = \frac{-3+\sqrt{3}}{6} \text{ and } \alpha_3 = 0. \]
		We obtain
		\[ C_{1,1} = \left(1+\frac{2}{\sqrt{3}}\right) F(\lambda_1)=1.256\dots, \; C_{2,1} = \left(1-\frac{1}{\sqrt{3}}\right) F(\lambda_2) = -0.048\dots \text{ and } C_{3,1} = 0 \]
		as well as
		\[ C_{1,2} = 0.251\dots, \; C_{2,2} = -0.017 \dots \text{ and } C_{3,2} = 0. \]
		We infer
		\begin{align*} \mathbb{E} \Big[ \big \Vert \sum_{n \le x} f(n) \big \Vert_{HS}^2 \Big] &= x \left( C_{1,1}  (\log x)^{\lambda_1-1} + C_{2,1} (\log x)^{\lambda_2 - 1} + C_{1,2} (\log x)^{\lambda_1-2} + C_{2,2} (\log x)^{\lambda_2-2} \right) \\ &+ O \left( x (\log x)^{\lambda_1-3} \right).
		\end{align*}
	\end{example}
	
	\subsection{The Non-Diagonalizable Case}
	
	If $T$ is not diagonalizable, it turns out that we need to find an estimate for a more difficult quantity, and we are only able to prove an ineffective asymptotic. More precisely, we need the following
	
	\begin{proposition}\label{AsympNonDiag}
		For fixed $z \in \mathbb{C}\setminus \mathbb{Z}^-$ and $r \in \mathbb{N}_0$, we have
		\begin{equation}
		\sum_{n \le x} \mu^2(n) \omega(n)^r z^{\omega(n)} \sim z^r F(z) x (\log x)^{z-1} (\log_2 x)^r.
		\end{equation}
	\end{proposition}

	\begin{proof}
	The idea is to find an asymptotic as $s \to 1$ for the associated Dirichlet series and then to apply Delange's Theorem \cite[Th\'{e}or\`{e}me IV]{Delange1}, compare also \cite[Theorem 7.28]{Tenenbaum1}.
	
	A crucial point in proving this is that
	\[  
	\sum_{n \ge 1} \frac{\mu^2(n) \omega(n) (\omega(n)-1) \cdots (\omega(n)-r+1) z^{\omega(n)}}{n^s} = z^r \frac{d^r}{dz^r} \left( \sum_{n \ge 1} \frac{\mu^2(n) z^{\omega(n)}}{n^s}\right).
	\]
	Expanding the falling factorials using the Stirling numbers of the second kind, denoted by curly brackets, we obtain
	\[
	\sum_{n \ge 1} \frac{\mu^2(n) \omega(n)^r z^{\omega(n)}}{n^s} = \sum_{k=0}^r \stirling{r}{k} z^k \frac{d^k}{dz^k} \left( \sum_{n \ge 1} \frac{\mu^2(n) z^{\omega(n)}}{n^s} \right).
	\]
	But we have
	\[ \frac{d}{dz} \left( \sum_{n \ge 1} \frac{\mu^2(n) z^{\omega(n)}}{n^s} \right) = \frac{d}{dz} (\zeta(s)^z F(s,z)) = (\log \zeta(s)) \zeta(s)^z F(s,z) + \zeta(s)^z F_z(s,z).  \]
	Here, $F_z$, denotes the derivative of $F$ in the second component. Inductively we obtain expansions of the form
	\[ 
	\frac{d^k}{dz^k} \left( \sum_{n \ge 1} \frac{\mu^2(n) z^{\omega(n)}}{n^s} \right) = \zeta(s)^z \big[(\log \zeta(s))^k F(s,z) + \dots \big],
	\]
	where the other terms involve lower powers of $\log \zeta(s)$ as well as derivatives of $F$. We thus obtain an expansion of the form
	
	\[ 
	\sum_{n \ge 1} \frac{\mu^2(n) \omega(n)^r z^{\omega(n)}}{n^s} = z^r \zeta(s)^z \big[ (\log \zeta(s))^r F(s,z) + \dots \big],
	\]
	where again the other terms involve lower (non-negative, integral) powers of $\log \zeta(s)$ multiplied by functions of $s$ and $z$ which are holomorphic around $s=1$ for any $z$. We are thus in a position to apply Delange's Theorem, which indeed implies that
	\[
	\sum_{n \le x} \mu^2(n) \omega(n)^r z^{\omega(n)} \sim z^r F(z) x (\log x)^{z-1} (\log_2 x)^r
	\]
	when $z \in \mathbb{C} \setminus \mathbb{Z}^-$, as claimed.
\end{proof}
	
	This Proposition allows us to prove the following
	
	\begin{theorem}\label{nondiag}
		Let $d \ge 1$ be an integer, let $X$ be a $\mathbb{C}^{d \times d}$-valued random variable and let $f$ be the associated matrix-valued multiplicative function. Suppose that $\mathbb{E} X = 0$ and
		\[\mathbb{E}\left[ \Vert X \Vert_{HS}^2 \mathbbm{1}(\Vert X \Vert_{HS}^2 > R) \right] \xrightarrow{R \to \infty} 0.\] 
		Define
		\begin{align*}
		T : \mathbb{C}^{d \times d} &\to \mathbb{C}^{d \times d},\\
		A &\mapsto \mathbb{E}[ X^* A X ],
		\end{align*}
		and let $\lambda_1, \dots, \lambda_t$ be the (distinct) eigenvalues of $T$ arranged in descending order according to their real parts. Let $p_T$ be the characteristic polynomial of $T$ and let  $c_1, \dots, c_l$ and $m_1, \dots, m_t$ be such that
		\[ p_T(x) = x^l + c_1 x^{l-1} + \dots + c_l = (x-\lambda_1)^{m_1} \cdots (x-\lambda_t)^{m_t}.\]
		Further, define
		\[ a_n := \mathbb{E} \left[ \Vert X_1 \cdots X_n \Vert_{HS}^{2k} \right], \]
		where $X_1, X_2, \dots$ are i.i.d. copies of $X$, and let $g_1 ,\dots, g_t$ be the polynomials satisfying $d_i := \deg g_i < m_i$ and
		\[ a_n = g_1(n) \lambda_1^n + \dots + g_t(n) \lambda_t^n \] 
		(see Theorem \ref{LinRec} and Lemma \ref{GenLinRec}). Let $R$ be the maximal real part among those $\lambda_i$ with $d_i > 0$. Define $L_1, L_2$ and $L_3$ to be the collection of $i$ such that $\Re \lambda_i > R, \Re \lambda_i = R$ and $\Re \lambda_i < R$, respectively. Lastly, let $d_{max} = \max_{i \in L_2} d_i$ and $L_2' = \{i \in L_2 : d_i = d_{max} \}$. Then for any $N \in \mathbb{N}$ there are explicit constants $C_{i,m}$ for $i \in L_1, \, m=1, \dots, N$ and $C_j'$ for $j \in L_2'$ such that
		\begin{align*} \mathbb{E} \Big[ \big \Vert \sum_{n \le x} f(n) \big \Vert_{HS}^2 \Big] = x \sum_{m=1}^N \sum_{i \in L_1} C_{i,m} (\log x)^{\lambda_i-m} &+ (1+o(1)) \sum_{j \in L_2'} C_j' x (\log x)^{\lambda_j-1} (\log_2 x)^{d_{max}} \\ &+ O \left( x (\log x)^{\lambda_1 - N - 1} \right) .
		\end{align*}
	\end{theorem}

	\begin{proof}
		Writing $b_i$ for the leading coefficient of $g_i$, the same argument as in Theorem \ref{diag} implies
		\begin{align*} 
		\mathbb{E} \Big[ \big \Vert \sum_{n \le x} f(n) \big \Vert_{HS}^2 \Big] &= \sum_{n \le x} \mu^2(n) \sum_{i=1}^t g_i(\omega(n)) \lambda_i^{\omega(n)} \\ &= \sum_{i \in L_1} b_i \sum_{n \le x} \mu^2(n) \lambda_i^{\omega(n)} + \sum_{i \in L_2'} b_i \sum_{n \le x} \mu^2(n) \omega(n)^{d_{max}} \lambda_i^{\omega(n)} \\ &+ O \left( \sum_{i \in L_2} \sum_{n \le x} \mu^2(n) \omega(n)^{d_{max}-1} \lambda_i^{\omega(n)} \right) + O_\varepsilon \left( \sum_{i \in L_3} \sum_{n \le x} \mu^2(n) (\lambda_i+\varepsilon)^{\omega(n)} \right) .
		\end{align*}
		Regarding the first summand, Proposition \ref{AsympDiag} directly tells us that for any $N \in \mathbb{N}$ there are explicit constants $C_{i,m}$ such that 
		\[ \sum_{i \in L_1} b_i \sum_{n \le x} \mu^2(n) \lambda_i^{\omega(n)} = x \sum_{m=1}^N \sum_{i \in L_1} C_{i,m} (\log x)^{\lambda_i - m} + O \left( x (\log x)^{\max \Re \lambda_i - N - 1} \right). \]
		Using Proposition \ref{AsympNonDiag} on the second summand directly implies
		\[ \sum_{j \in L_2'} b_j \sum_{n \le x} \mu^2(n) \omega(n)^{d_{max}} \lambda_j^{\omega(n)} = (1+o(1)) \sum_{j \in L_2'} C_j' x (\log x)^{\lambda_j-1} (\log_2 x)^{d_{max}} \]
		for some explicit constants $C_j'$.
		
		Proposition \ref{AsympNonDiag} furthermore implies
		\[ \sum_{i \in L_2} \sum_{n \le x} \mu^2(n) \omega(n)^{d_{max}-1} \lambda_i^{\omega(n)} = O \left( x (\log x)^{R-1} (\log_2 x)^{d_{max}-1} \right) = o \left( x (\log x)^{R-1} (\log_2 x)^{d_{max}} \right). \]
		For the last error term, fix $\varepsilon > 0$ such that $\lambda_i + \varepsilon < R$ for all $i \in L_3$. Then
		\[ \sum_{i \in L_3} \sum_{n \le x} \mu^2(n) (\lambda_i+\varepsilon)^{\omega(n)} = o \left( x (\log x)^{R-1} \right) \]
		and the claim follows.
	\end{proof}
	
	\section{An Upper Bound for Higher Even Moments}\label{HighMom}
	
		Let $s \ge 1 $, and let $X,X_1, X_2, \dots$ be i.i.d. $\mathbb{C}^{d \times d}$-valued random variables with \[\mathbb{E}[\Vert X \Vert_{HS}^s] < \infty.\] Then 
		\[\rho_s := \rho_s(X) := \lim_{n \to \infty} \mathbb{E}\left[ \Vert X_1 \cdots X_n \Vert_{HS}^s \right]^{\frac{1}{s n}} \]
		will be called the spectral $s$-radius of $X$. If $S \subset \mathbb{C}^{d \times d}$ is bounded then
		\[ \rho_\infty(S) := \lim_{k \to \infty} \sup \{ \Vert A_{i_1} \cdots A_{i_k} \Vert_{HS}^{1/k} \; : \; A_i \in S \} \]
		is called the joint spectral radius of $S$. Note that all these quantities are in fact independent of the chosen norm, since all norms on $\mathbb{C}^{d \times d}$ are equivalent.
		
		The joint spectral radius has been studied in great detail in contexts such as dynamical systems, wavelets, optimization and control. We refer the interested reader to \cite{Jungers1}. For the generalized joint spectral radius, its geometric interpretation and relation to Kronecker products, see e.g. \cite{Protasov1,Protasov2}.
		
		We note at this point that by H\"older's inequality, $\rho_s$ is monotonically increasing and if $X$ is the uniform distribution on a bounded set $S$ then we have $\rho_s \uparrow \rho_\infty$ as $s \to \infty$. Also, note that $\rho_{2 k} = \lambda_1^{1/2 k}$, where $\lambda_1 \ge 0$ is as in Theorem \ref{nondiag}.
	
	The goal of this section is to prove the following
	
	\begin{theorem}
		Let $k \ge 2$ be an integer, and let $X$ be a symmetric $\mathbb{C}^{d \times d}$-valued random variable satisfying $\mathbb{E} X = 0$ and
		\[ \mathbb{E} \left[ \Vert X \Vert_{HS}^{2k} \mathbbm{1}( \Vert X \Vert_{HS}^{2k} >R )  \right] \to 0 \]
		as $R \to \infty$. Let $f$ be the random matrix-valued multiplicative function associated to $X$.
		Then we have
		\begin{equation}\label{HigherMoment} \mathbb{E} \left[ \Big \Vert \sum_{n \le x} f(n) \Big \Vert_{HS}^{2 k} \right] \ll x^k (\log x)^{[\rho_{2 k}^2+1] \binom{2 k}{2}-2 k},  \end{equation}
		where $[ \, \cdot \, ]$ denotes the integral part.
	\end{theorem}

	\begin{proof}
		Denoting by $\Box$ a generic square, we have
		\begin{align*}
		\mathbb{E} \left[ \Big \Vert \sum_{n \le x} f(n) \Big \Vert_{HS}^{2 k} \right] &= \mathbb{E} \left[ \mathrm{Tr}\left( \sum_{n_1, n_2 \le x} f(n_1)^* f(n_2) \right)^k \right] \\ &= \sum_{n_1, \dots, n_{2 k} \le x} \mathrm{Tr} \left( \mathbb{E} \left[ f(n_1)^* f(n_2) \otimes \dots \otimes f(n_{2 k-1})^* f(n_{2 k}) \right] \right) \\
		&= \sum_{\substack{n_1, \dots, n_{2 k} \le x \\ n_1 \cdots n_{2 k} = \Box }} \mathrm{Tr} \left( \mathbb{E} \left[ (f(n_1) \otimes f(n_3) \otimes \dots \otimes f(n_{2 k-1}))^* (f(n_2) \otimes \dots f(n_{2 k}) \right] \right) \\
		&\le \sum_{\substack{ n_1, \dots, n_{2 k} \le x \\ n_1\cdots n_{2 k} = \Box }} \mathbb{E} \left[ \Vert f(n_1) \otimes \dots \otimes f(n_{2 k-1}) \Vert_{HS} \Vert f(n_2) \otimes \dots \otimes f(n_{2 k}) \Vert_{HS} \right] \\
		&\le \sum_{\substack{ n_1, \dots, n_{2 k} \le x \\ n_1\cdots n_{2 k} = \Box }} \mathbb{E}[\Vert f(n_1) \Vert_{HS}^{2 k}]^{1/2k} \cdots \mathbb{E}[\Vert f(n_{2 k}) \Vert_{HS}^{2 k}]^{1/2k}.
		\end{align*}
		But from the definition of $\rho_{2 k}$, we see that
		\begin{equation}\label{JSRB} \mathbb{E}[\Vert f(n) \Vert_{HS}^{2 k}]^{1/2k} \ll_\varepsilon \mu^2(n) (\rho_{2 k}+\varepsilon)^{\omega(n)} \end{equation}
		and in particular
		\[ \mathbb{E}[\Vert f(n) \Vert_{HS}^{2 k}]^{1/2k} \ll \mu^2(n) [\rho_{2 k}^2+1]^{\omega(n)/2}. \]
		We thus obtain
		\[ \mathbb{E} \left[ \Big \Vert \sum_{n \le x} f(n) \Big \Vert_{HS}^{2 k} \right] \ll \sum_{\substack{ n_1, \dots, n_{2 k} \le x \\ n_1\cdots n_{2 k} = \Box }} \mu^2(n_1) \cdots \mu^2(n_{2 k}) [\rho_{2 k}^2+1]^{(\omega(n_1) + \dots + \omega(n_{2 k}))/2}. \]
		It thus remains to prove that
		\begin{equation}\label{MultDSBound} \sum_{\substack{ n_1, \dots, n_{2 k} \le x \\ n_1\cdots n_{2 k} = \Box }} \mu^2(n_1) \cdots \mu^2(n_{2 k}) m^{(\omega(n_1) + \dots + \omega(n_{2 k}))/2} \ll x^k (\log x)^{m \binom{2k}{2}-2 k} \end{equation}
		for all $m \in \mathbb{N}$. We proceed similar to the proof of \cite[Theorem 4]{HarperNR1}. To this end, let $g$ be the multiplicative function supported on squarefree integers such that $g(n_1,\dots,n_{2 k}) = m^{(\omega(n_1) + \dots + \omega(n_{2 k}))/2}$ when $n_1 \cdots n_{2 k}$ is a square, and $0$ otherwise. Then the associated multiple Dirichlet series
		\[ G(s) := \sum_{d_1, \dots, d_{2 k} \ge 1} \frac{g(d_1, \dots, d_{2 k})}{d_1^{s_1} \cdots d_{2 k}^{s_{2 k}}} \]
		has the Euler product representation
		\[ G(s) = \prod_p \sum_{\substack{0 \le \alpha_1, \dots, \alpha_{2 k} \le 1 \\ \alpha_1 + \dots + \alpha_{2 k} \;\equiv \;0 \; (2)}} \frac{m^{(\alpha_1 + \dots + \alpha_{2 k})/2}}{p^{\alpha_1 s_1 + \dots + \alpha_{2 k} s_{2 k}}}. \]
		This factors as
		\[ H(s_1,\dots,s_{2 k}) \prod_{1 \le i < j \le 2k} \zeta(s_i + s_j)^m \]
		with $H$ being holomorphic strictly to the left of $s= (\frac12, \dots, \frac12)$. The claim follows from \cite[Theorem 2]{DeLaBreteche1}, choosing each of the linear forms $s_i+s_j$ for $1 \le i < j \le 2 k$  precisely $m$ times, so that they are $m \binom{2k}{2}$ and have rank $2 k$ (this is where we are using that $k \ge 2$).
	\end{proof}

	\begin{remark}
		Note that our argument in fact implies a stronger statement than (\ref{MultDSBound}), namely that for fixed $m \in \mathbb{N}$ we have
		\[
		\sum_{\substack{ n_1, \dots, n_{2 k} \le x \\ n_1\cdots n_{2 k} = \Box }} \mu^2(n_1) \cdots \mu^2(n_{2 k}) m^{(\omega(n_1) + \dots + \omega(n_{2 k}))/2} \sim C_k x^k (\log x)^{m \binom{2k}{2}-2 k}.
		\]
		It seems rather natural, also in light of Proposition \ref{AsympDiag}, to conjecture that this asymptotic holds for all fixed $z > 0$ (say) in place of $m \in \mathbb{N}$. However, this is not possible when $z$ is small: Looking only at the contribution of tuples $(n_1, \dots, n_{2k}) = (p_1, \dots, p_k, p_1, \dots, p_k)$, we see that for any fixed $z > 0$ we have
		\[
		\sum_{\substack{ n_1, \dots, n_{2 k} \le x \\ n_1\cdots n_{2 k} = \Box }} \mu^2(n_1) \cdots \mu^2(n_{2 k}) z^{(\omega(n_1) + \dots + \omega(n_{2 k}))/2} \gg \frac{x^k}{(\log x)^k}.
		\]
		When $z$ is sufficiently small (depending only on $k$) then this is clearly a contradiction. It would be very interesting to know what the correct asymptotic for this expression is, or more generally for any multiple Dirichlet series of this type, i.e. to have a generalisation of \cite[Theorem 2]{DeLaBreteche1} to poles of non-integral order. Our remark here suggests that this is not as straightforward as one might expect.
		
		Note also that Theorem \ref{LinRec} implies that we can improve (\ref{JSRB}) to
		\[ \tag{$*$} \mathbb{E}[\Vert f(n) \Vert_{HS}^{2 k}]^{1/2k} \ll \mu^2(n) \omega(n)^{r/2k} \rho_{2 k}^{\omega(n)}, \]
		where $0 \le r < l^\mathbb{C}$ is the degree of $g_1$. This leads in a natural way to the even more general question of obtaining an asymptotic (or upper bound) for multiple Dirichlet series with a pole of non-integral order times a logarithmic pole.
		
		In particular, ($*$) implies that if $T$ is diagonalizable (or more generally if $\deg g_1 = 0$) then we get
		\[ \mathbb{E}[\Vert f(n) \Vert_{HS}^{2 k}]^{1/2k} \ll \mu^2(n)\rho_{2 k}^{\omega(n)}.\]
		If in addition $\rho_{2 k}^2$ is an integer, our argument thus gives
		\[\mathbb{E} \left[ \Big \Vert \sum_{n \le x} f(n) \Big \Vert_{HS}^{2 k} \right] \ll x^k (\log x)^{\rho_{2 k}^2 \binom{2 k}{2}-2 k} \]
		in place of (\ref{HigherMoment}). In particular, if $f$ is a Rademacher multiplicative function then $\rho_{2 k} = 1$ for all $k$ and up to constant we obtain the optimal upper bound. Noting that all our inequalities in the proof are in fact equalities in this case and that (\ref{MultDSBound}) can be improved to an asymptotic, we can recover \cite[Theorem 4]{HarperNR1}, but this leads to the identical argument as it is carried out there.
		
		It would be interesting to know if one can obtain a lower bound for the higher moments, for example in terms of the joint spectral subradius.
	\end{remark}

	\begin{example}
		We continue with example \ref{SL2}. We were not able to find explicit expressions for $\rho_{2 k}$ when $k \ge 2$; it seems quite plausible that such expressions don't exist. However, we can bound $\rho_{2 k}$ from above by the joint spectral radius $\rho_\infty(S)$. Moreover, using the JSR toolbox for Matlab (see \cite{VHJ} for its documentation and instructions for installation), we could compute that $\rho_\infty^2 = 1.8173540 \dots < 2$. In particular, we see that
		\[ \mathbb{E} \left[ \Big \Vert \sum_{n \le x} f(n) \Big \Vert_{HS}^{2 k} \right] \ll x^k (\log x)^{4k(k-1)} \]
		holds for $k \ge 2$ and $x \ge 2$.
	\end{example}
	
	\bibliography{Eigenvalues}
	
\end{document}